\documentclass[11pt]{article}

\usepackage{amsfonts}
\usepackage{amsmath}

\newcommand{\beq}{\begin{equation}}
\newcommand{\eeq}{\end{equation}}
\newcommand{\bea}{\begin{eqnarray}}
\newcommand{\eea}{\end{eqnarray}}
\newcommand{\beas}{\begin{eqnarray*}}
\newcommand{\eeas}{\end{eqnarray*}}

\newcommand{\G}{\mathbb{G}}
\newcommand{\Ga}{\Gamma}
\newtheorem{theorem}{Theorem}[section]

\newtheorem{proposition}[theorem]{Proposition}
\newtheorem{corollary}[theorem]{Corollary}
\newtheorem{lemma}[theorem]{Lemma}
\newtheorem{remark}[theorem]{Remark}
\newtheorem{example}[theorem]{Example}
\newtheorem{examples}[theorem]{Examples}
\newtheorem{foo}[theorem]{Remarks}

\newenvironment{proof}{\addvspace{\medskipamount}\par\noindent{\it
Proof}.}
{\unskip\nobreak\hfill$\Box$\par\addvspace{\medskipamount}}

\newcommand{\p}{\partial}

\newcommand{\bM}{\mathbb M}

\newcommand{\F}{\mathcal F}
\newcommand{\Ho}{\mathcal H}

\newcommand{\R}{\mathbb R}

\title{Reverse Poincar\'e inequalities, Isoperimetry, and Riesz transforms in Carnot groups}

\author{Fabrice Baudoin, Michel Bonnefont}

\begin{document}

\maketitle

\begin{abstract}
We prove an optimal reverse Poincar\'e inequality for the heat semigroup generated by the sub-Laplacian on a Carnot group of any step.  As an application we give new proofs of the isoperimetric inequality and of the boundedness of the Riesz transform in Carnot groups.
\end{abstract}

\tableofcontents

\section{Introduction}

In this paper, we are interested in reverse Poincar\'e inequalities for Markov semigroups. If $P_t$ is a Markov semigroup generated by a diffusion operator $L$, such an inequality writes in the form of
\begin{align}\label{rp}
\Gamma(P_t f) \le C(t) ( P_t f^2-(P_tf)^2)
\end{align}
where $C(t)$ is a positive constant and  $\Gamma$ the \textit{carr\'e du champ} operator: $$\Gamma(f)=\frac{1}{2}(Lf^2 -2fLf).$$ This inequality is a regularization estimate in the sense that it allows to control derivatives of $P_t f$ in terms of the sup-norm of $f$ only. If $L$ is the Laplace operator on $\mathbb{R}^n$, it is easy to check that the inequality \eqref{rp} holds with $C(t)=\frac{1}{2t}$, and this is the best possible constant.  

\

The reverse Poincar\'e inequality is closely related to intrinsic curvature properties of the generator $L$. Actually, it turns out that the inequality 
\[
\Gamma(P_t f) \le \frac{K}{e^{2Kt}-1} ( P_t f^2-(P_tf)^2)
\]
is equivalent to the Bakry-\'Emery criterion
\[
\Gamma_2(f) \ge K \Gamma(f),
\]
where
\[
\Gamma_2(f)=\frac{1}{2} (L\Gamma(f) -2\Gamma(f,Lf)),
\]
is the usual Bakry's $\Gamma_2$ operator. For further details on this, we refer to the book by Bakry-Gentil-Ledoux \cite{BGL}. As an example, on a complete Riemannian manifold, the inequality 
\[
\| \nabla P_t f \|^2 \le \frac{K}{e^{2Kt}-1} ( P_t f^2-(P_tf)^2)
\]
is satisfied for the semigroup generated by the Laplace-Beltrami operator, if and only if the Ricci curvature of the manifold is bounded from below by $K$. 

\

As is it now understood, the Bakry-\'Emery criterion requires some form of ellipticity for the generator and typically fails to hold for strictly hypoelliptic operators \cite{BBBC,BB}. However. if $L$ is a hypoelliptic diffusion operator, the semigroup $P_t$ is smoothing in the sense that it transforms bounded Borel functions into smooth functions. For this reason, it is natural to expect, that under suitable conditions an inequality like \eqref{rp} should hold. In \cite{BBBC}, it was proved that if $L$ is the sub-Laplacian in the Heisenberg group, the following reverse Poincar\'e inequality holds
\[
\Gamma(P_t f) \le \frac{1}{t} ( P_t f^2-(P_tf)^2).
\]
Moreover, the constant $\frac{1}{t}$ is optimal. In the same reference, it was conjectured that a similar inequality should hold on any Carnot group. In the present paper, we prove that this is indeed the case and moreover compute the optimal constant of the inequality.

\

The paper is organized as follows. In Section 2, we first recall some basic results about Carnot groups and their heat semigroups and then proceed to work out the optimal reverse Poincar\'e inequality. In Section 3, we show how the reverse Poincar\'e inequality implies the isoperimetric inequality. Finally in Section 4, we give a second application of the reverse Poincar\'e inquality, by showing that the Riesz transform on Carnot groups is bounded.

\section{The optimal reverse Poincar\'e inequality for the heat semigroup in Carnot groups}

\subsection{Preliminaries on Carnot groups}

A Carnot group of step (or depth) $N$ is a simply connected Lie
group $\mathbb{G}$ whose Lie algebra can be written
\[
\mathfrak{g}=\mathcal{V}_{1}\oplus...\oplus \mathcal{V}_{N},
\]
where
\[
\lbrack \mathcal{V}_{i},\mathcal{V}_{j}]=\mathcal{V}_{i+j}
\]
and
\[
\mathcal{V}_{s}=0,\text{ for }s>N.
\]
From the above properties, it is of course seen that Carnot groups are nilpotent.
The number
\[
Q=\sum_{i=1}^N i \dim \mathcal{V}_{i}
\]
is called the homogeneous dimension of $\mathbb{G}$.

\begin{example}

\

\begin{enumerate}
\item (Commutative case) The group $\left( \mathbb{R}^d ,+ \right)$ is the only commutative
Carnot group.

\item (Heisenberg groups) Consider the set $\mathbb{H}_n =\mathbb{R}^{2n} \times \mathbb{R}$
endowed with the group law
\[
(x,\alpha) \star (y, \beta)=\left( x+y, \alpha + \beta
+\frac{1}{2} \omega (x,y) \right),
\]
where $\omega$ is the standard symplectic form on
$\mathbb{R}^{2n}$, that is
\[
\omega(x,y)= x^t \left(
\begin{array}{ll}
0 & -\mathbf{I}_{n} \\
\mathbf{I}_{n} & ~~~0
\end{array}
\right) y.
\]
On
$\mathfrak{h}_n$ the Lie bracket is given by
\[
[ (x,\alpha) , (y, \beta) ]=\left( 0, \omega (x,y) \right),
\]
and it is easily seen that
\[
\mathfrak{h}_n=\mathcal{V}_1 \oplus \mathcal{V}_2,
\]
where $\mathcal{V}_1 =\mathbb{R}^{2n} \times \{ 0 \}$ and
$\mathcal{V}_2= \{ 0 \} \times \mathbb{R}$. Therefore
$\mathbb{H}_n$ is a Carnot group of depth 2.

\item (Engel group) The Engel group is the $4$-dimensional Lie group of matrices
\[
\mathbb{E}= \left\{ 
\left(
\begin{array}{llll}
1 & x & \frac{x^2}{2} & z \\
0 & 1 & x & w \\
0 & 0 & 1 & y \\
0 & 0 & 0 & 1
\end{array}
\right), x,y,w, z \in \mathbb{R} \right\}.
\]
Its Lie algebra $\mathfrak{e}$ is generated by the matrices
\[
X= \left(
\begin{array}{llll}
0 & 1 & 0 & 0 \\
0 & 0 & 1 & 0 \\
0 & 0 & 0 & 0 \\
0 & 0 & 0 & 0
\end{array}
\right),
Y= \left(
\begin{array}{llll}
0 & 0 & 0 & 0 \\
0 & 0 & 0 & 0 \\
0 & 0 & 0 & 1 \\
0 & 0 & 0 & 0
\end{array}
\right)
\]
\[
W= \left(
\begin{array}{llll}
0 & 0 & 0 & 0 \\
0 & 0 & 0 & 1 \\
0 & 0 & 0 & 0 \\
0 & 0 & 0 & 0
\end{array}
\right),
Z= \left(
\begin{array}{llll}
0 & 0 & 0 & 1 \\
0 & 0 & 0 & 0 \\
0 & 0 & 0 & 0 \\
0 & 0 & 0 & 0
\end{array}
\right),
\]
for which we have the following structure relations,
\[
[X,Y]=W, \quad [X,W]=Z
\]
and all other brackets are zero. The Engel group is easily seen to be a Carnot group of step 3.
\end{enumerate}
\end{example}
On $\mathfrak{g}$ we can
consider the family of linear operators $\delta_{t}:\mathfrak{g}
\rightarrow \mathfrak{g}$, $t \geq 0$ which act by scalar
multiplication $t^{i}$ on $\mathcal{V}_{i} $. These operators are
Lie algebra automorphisms due to the grading. The maps $\delta_t$
induce Lie group automorphisms $\Delta_t :\mathbb{G} \rightarrow
\mathbb{G}$ which are called the canonical dilations of
$\mathbb{G}$. It is easily seen that there exists on $\mathbb{G}$ a complete and smooth vector field $D$ such that
\[
\Delta_t =e^{(\ln t) D}.
\]
This vector field $D$ is called the dilation vector field on $\mathbb{G}$. If $X$ is a left (or right) invariant smooth horizontal vector field on $\mathbb{G}$, that is $X(0) \in \mathcal{V}_{1}$, we have for every $f \in C^\infty(\mathbb{G})$, and $t \ge 0$,
\[
X(f \circ \Delta_t)=t Xf.
\]

\

Let us now pick a basis $V_1,...,V_d$ of the vector
space $\mathcal{V}_1$. The vectors $V_i$'s can be seen as left
invariant vector fields on $\mathbb{G}$. In the sequel, these vector fields shall still be denoted by $V_1,...,V_d$ and the corresponding right invariant vector fields shall be denoted by $\hat{V}_1,...,\hat{V}_d$. The left invariant sub-Laplacian on $\mathbb{G}$ is the operator:
\[
L=\sum_{i=1}^d V_i^2.
\]
It is essentially self-adjoint on the space of smooth and compactly supported function $f: \mathbb{G} \rightarrow \mathbb{R}$ with the respect to the Haar measure $\mu$ of $\mathbb{G}$. The heat semigroup $(P_t)_{t\ge 0}$ on $\mathbb{G}$, defined through the spectral theorem, is then seen to be a Markov semigroup (see \cite{varopoulos}). This heat semigroup admits a positive fundamental solution named the heat kernel and  denoted by $p_t(g,g')$. We often simply denote $p_t(g)=p_t(0,g)$ for the heat kernel issued from the identity. By left invariance, it is enough to know this heat kernel issued from the identity to recover all the heat kernels. Since Carnot groups enjoy the volume doubling property and carry the Poincar\'e inequality on balls, we deduce that $p_t$ satisfies the double-sided Gaussian bounds (see Theorem 2.9 in \cite{varopoulos}):

\begin{align}\label{bound1}
\frac{C^{-1}}{t^{Q/2}} \exp
\left(-\frac{ C_1d(0,g)^2}{t}\right)\le p_t(g)\le \frac{C}{t^{Q/2}} \exp
\left(-C_2\frac{d(0,g)^2}{t}\right),
\end{align}
for some constants $C,C_1,C_2 >0$. Here $d(0,g)$ is Carnot-Carath\'eodory distance from $0$ to $g$ in $\G$.

As usual, if $f: \mathbb{G} \rightarrow \mathbb{R}$ is a smooth function, we denote, 
\[
\Gamma (f,f)=\sum_{i=1}^d (V_i f)^2.
\]
This is the \textit{carr\'e du champ} operator of $L$. Sometimes, as a shortened notation we will denote $\Gamma(f)$ for $\Gamma(f,f)$. The following gradient bound can be found in \cite{varopoulos}, Theorem 2.7,
\begin{align}\label{bound2}
\sqrt{\Gamma(p_t)} (g) \le \frac{C}{t^{\frac{Q+1}{2}}} \exp
\left(-C_3\frac{d(0,g)^2}{t}\right).
\end{align}

We also introduce the right-invariant sub-Laplacian:
\[
\hat L=\sum_{i=1}^d \hat V_i^2
\]
and we denote by $\hat P_t $ the associated heat semigroup and by $\hat p_t$ the associated heat kernel.
We also denote
\[
\hat \Gamma (f,f)=\sum_{i=1}^d (\hat V_i f)^2.
\]

First, we begin with two useful lemmas, whose proofs are classical and let to the reader. The main argument is that for $f  \in L^2 (\mathbb{G},\mu)$, $P_t f$ is the unique solution in $L^2(\mathbb{G},\mu)$ of the parabolic Cauchy problem:
\begin{align*}
\begin{cases}
\frac{\partial \phi}{\partial t}=L \phi \\
\phi (0,x)=f(x).
\end{cases}
\end{align*}
The same characterization holds for $\hat P_t $.  Our first lemma relies the two semigroups $P_t$ and $\hat P_t$. 
\begin{lemma}\label{hat-P_t}
Let $f \in L^2 (\mathbb{G},\mu)$. Then for $g \in \mathbb{G}$, one has:
$$
\hat P_t (f) (g)= P_t \left(f \circ Ad(g^{-1})\right) (g)
$$
where $Ad(g^{-1})$ is the function defined by $Ad(g^{-1}) (h)= g^{-1} h  g$.
As a consequence, one has:
$$
\hat P_t(f)(0)=P_t(f)(0),
$$
and for every $g\in \mathbb{G}$
$$
\hat  p_t (g) = p_t (g).
$$
\end{lemma}

The second lemma illustrates the scaling property of the semigroup $P_t$  with respect to the dilations. The second identity of the lemma can be obtained by differentiating the first one at $c=1$.

\begin{lemma}
Let $f \in  L^2 (\mathbb{G},\mu)$. For every $t,c \ge 0$,
\[
P_t( f\circ \Delta_{\sqrt{c} })=(P_{ct} f ) \circ \Delta_{\sqrt{c}} .
\]
Moreover, if $f \in C_0^\infty(\G)$, then for every $t \ge 0$,
\[
P_t Df=DP_tf + t P_t Lf.
\]
\end{lemma}

We conclude this preliminary section with an integrability lemma which shall be useful in the sequel.

\begin{lemma}\label{integrability}
Let $f: \G \to \R$ be a smooth function with polynomial growth, that is,
\[
| f(g) | \le C (1+ d(0,g))^N, \quad g \in \G,
\]
for some $C >0, N\ge 0$, then for  $t >0$,
\[
\int_\G f p_t \Gamma (\ln p_t, \ln p_t) d\mu <+\infty.
\]
\end{lemma}

\begin{proof}
 As the dilation vector field $D$ vanishes at 0,
  for all $t>0$ and  for any smooth and compactly supported $h$,
  $$
  P_{t}((tL-D)h)(0)=0,
  $$ 
  that is 
  $$\int_\G (tL-D)h \; p_t\,  d\mu=0.
  $$
  This implies
  $$
  \left(tL+D+\frac{Q}{2}\right)p_t=0
  $$
  because the the adjoint $D^*$ of $D$ is  $- D-\frac{Q}{2}$.
  
  \
  
  Let now $h \in C_0^\infty(\G)$. We have
  \begin{align*}
  \int_\G h f p_t \Gamma (\ln p_t, \ln p_t) d\mu &=  \int_\G h f  \Gamma (\ln p_t,  p_t) d\mu \\
   &= \int_\G  \Gamma ( hf \ln p_t,  p_t) d\mu- \int_\G \ln p_t  \Gamma (h f,  p_t) d\mu \\
   &=-\int_\G  hf \ln p_t Lp_t  d\mu- \int_\G \ln p_t  \Gamma (h f,  p_t) d\mu \\
   &=-\frac{1}{t} \int_\G D( hf \ln p_t ) p_t  d\mu- \int_\G \ln p_t  \Gamma (h f,  p_t) d\mu
  \end{align*}
  
  In exponential coordinates, the vector fields $D$ and $V_i$'s have polynomial coefficients. We can thus easily construct an increasing  sequence $h_n \in C_0^\infty(\G)$ such that $0 \le h_n \le 1$, $h_n \to 1$ and $| D h_n | (g) \le \frac{1}{n} P(g) $,  $\sqrt{\Gamma}(h_n) (g) \le \frac{1}{n} P(g)$, where $P$ is a function with polynomial growth on $\G$. We now use the previous equalities with $h_n$ in place of $h$. We have
  \[
   \int_\G D( h_n f \ln p_t ) p_t  d\mu= \int_\G h_n D( f \ln p_t ) p_t  d\mu+ \int_\G D( h_n) f p_t \ln p_t d\mu.
  \]
  Thus, from the bounds \eqref{bound1} and \eqref{bound2}, by dominated convergence, we obtain
  \[
   \int_\G D( h_n f \ln p_t ) p_t  d\mu \to \int_\G  D( f \ln p_t ) p_t  d\mu.
  \]
  Similarly, we have
  \[
  \int_\G \ln p_t  \Gamma (h_n f,  p_t) d\mu \to   \int_\G \ln p_t  \Gamma (f,  p_t) d\mu.
  \]
\end{proof}

\subsection{The optimal reverse Poincar\'e inequality}

We now turn to the inequality which is of interest for us. The following reverse Poincar\'e inequality for the heat semigroup is optimal.

\begin{proposition}\label{P-I-gen}
Let $f: \mathbb{G} \rightarrow \mathbb{R}$ be a smooth and compactly supported function. For $g \in \mathbb{G}$,
\[
\Gamma (P_t f,P_t f)(g) \le \frac{\Lambda}{t} \left( P_t f^2 (g)-(P_t f)^2 (g) \right), \quad t>0.
\]
where the constant $\Lambda$ is   the largest eigenvalue of the symmetric matrix 
$$
M=\left(  \int_\G \hat V_i (\ln p_1) \hat V_j(\ln p_1) p_1 d\mu \right)_{1\leq i,j \leq d}.
$$
Moreover, the constant $\Lambda$ is sharp. 
\end{proposition}

\begin{proof}
By left invariance and scaling, it is enough to check it at the identity and $t=1$.
Let $f: \mathbb{G} \rightarrow \mathbb{R}$ be a smooth and compactly supported function, then
$$
\Gamma (P_1 f,P_1 f)(0)= \sup_{\sum_{i=1}^d a_i^2=1} \left(\sum_{i=1}^d a _i V_i P_1(f)(0)\right)^2.
$$
Now, let $a_i\in \R$ be such that $\sum_{i=1}^d a_i^2=1$, then
  \begin{eqnarray*}
  \sum_{i=1}^d a _i V_i P_1(f)(0)&=&
  P_{1}( \sum_{i=1}^d a_i \hat V_i f)(0)\\ 
                                                &=&  -\int_\G  \sum_{i=1}^d a_i \hat V_i (p_1) \, f d\mu\\
                                                 &=&  -\int_\G  \sum_{i=1}^d a_i \hat V_i (\ln p_1) \, f \, p_1 d\mu
  \end{eqnarray*}
 Therefore, by Cauchy-Schwarz inequality against the measure $p_1 d\mu$, we have
 $$
 \left(\sum_{i=1}^d a _i V_i P_1(f)(0)\right)^2 \leq   \int_\G \left(\sum_{i=1}^d a_i \hat V_i ( \ln p_1)\right)^2 p_1 d\mu  \,  \int_\G f^2 p_1 d\mu.
 $$
 But one can write:
 \beas
 \int_\G \left(\sum_{i=1}^d a_i \hat V_i ( \ln p_1)\right)^2 p_1 d\mu &=& \sum_{i,j =1}^d  a_i a_j \int_\G  \hat V_i ( \ln p_1) \hat V_j(\ln p_1) p_1 d\mu \\
                                                                      &=& A^t \, M \,  A
 \eeas
 with $M$  the matrix defined in the proposition and $A=(a_1, \dots a_d)^t$. The result follows then easily 

\end{proof}

The following proposition gives a lower and upper bound for the optimal constant $\Lambda$.

\begin{proposition}
We have
\[
\frac{Q}{2d}\le \Lambda \le \frac{Q}{2}.
\]
\end{proposition}

\begin{proof}
Since the matrix $M$ is symmetric, the constant $\Lambda$ satisfies the following inequality:
$$
\frac{1}{d} \textrm{ trace } M \leq \Lambda \leq \textrm{ trace } M
$$
and the trace of $M$ is given by  
$$
\textrm{ trace } M = \int_\G \hat \Ga (\ln p_1) p_1 d\mu = \int_\G \Ga (\ln p_1) p_1 d\mu.
$$
Thus, we just need to prove that
$$
\int_\G \Gamma(\ln p_t) p_t d\mu = \frac{Q}{2 t}.
$$
To this end, recall that 
  $$
  \left(tL+D+\frac{Q}{2}\right)p_t=0.
  $$

   Multiply then  by $\ln p_t$ and
  taking integral gives
  $$
  \int_\G \ln p_t \; \left(tL+D+\frac{Q}{2}\right)p_t\, d\mu = 0.
  $$
   But by using an integration by parts (which we can justify as in Lemma \ref{integrability}), we have
   $$
    t\int_\G \ln p_t L p_t\, d\mu = - t \int_\G \Gamma(\ln p_t, \ln p_t)\;  p_t\, d\mu.
   $$
   
    Moreover, we have
  $$
  \int_\G \ln p_t\; \left(D+\frac{Q}{2}\right)p_t \,d\mu %
  = -\int_\G p_t \; D\ln p_t \,d\mu %
  = -\int_\G D p_t d\mu= \frac{Q}{2}\int_\G p_t \, d\mu %
  =\frac{Q}{2}
  $$
  and therefore
  $$
  \int_\G \Ga(\ln p_t, \ln p_t)\;  p_t \,d\mu = \frac{Q}{2t}.
  $$  
\end{proof}

\subsection{The case of H-type groups}

We now prove that the lower bound is achieved in a special class of Carnot groups, the so-called $H$-type groups.  For us, a $H$-type group will be simply $\R^{2n+m} =\R^{2n}\times \R^{m} $ equipped with the product
$$
v*w=v+w+ \frac{1}{2}[v,w]
$$  
where $[\cdot,\cdot]$ is a bracket operation on $\R^{2n+m}$ whose  center is  ${0}\times   \R^{m}$ and such the map $J_z:\R^{2n} \to \R^{2n}$ defined for $z\in \R^{m}$  by:
$$
<J_z(x),y>= <[x,y],z> \textrm{ for all } x,y \in \R^{2n}
$$
is orthogonal when $|z|=1$. In the above we have identified $x\in \R^{2n}$ with $(x,0)\in  \R^{2n}\times \R^{m}$ and $z\in  \R^{m} $ with $(0,z) \in \R^{2n}\times \R^{m}$ and $|\cdot|$ denotes the classical Euclidean norm.

The left-invariant vector fields which coïncide with $\frac{\partial}{\partial x_i}$ at the identity write:
$$
V_i = \frac{\partial}{\partial x_i} +\frac{1}{2} \sum_{j=1}^m <J_{u_j} x, e_i> \frac{\partial}{\partial z_j} 
$$
whereas the corresponding right-invariant vector fields  write:
$$
\hat V_i = \frac{\partial}{\partial x_i} -\frac{1}{2} \sum_{j=1}^m <J_{u_j} x, e_i> \frac{\partial}{\partial z_j} 
$$
for $i=1,\dots,2n$ and where $(e_1,\dots e_{2n})$ is the canonical basis of $\R^{2n}$ and $(u_1,\dots u_{m})$ the canonical basis of $\R^{m}$.
It is easy to see that it is a Carnot group of step 2. The Haar measure is just the Lebesgue measure and the heat kernel associated to $L=\sum_{i=1}^{2n} V_i^2$ issued from the identity is given by (see for instance \cite{eldredge}): 

$$
p_t(x,z)=\frac{1}{(2\pi)^m}\frac{1}{(4\pi)^n} \int_{\R^m} e^{i<\lambda,z>} e^{-\frac{|\lambda| |x|^2}{4} \coth |\lambda| t }
           \left( \frac{|\lambda|}{\sinh |\lambda |t}\right)^n d\lambda
$$
and therefore is only a function of $|x|$ and $|z|$.

\

In the case of $H$-type groups, we  obtain the more precise statement.
\begin{corollary}\label{P-I-H-type}
Assume $\G$ to be an $H$-type group, then 
\[
\Lambda=\frac{Q}{2d}.
\]
\end{corollary}
  
  \begin{proof}
Let $\G$ be a $H$-type group.
By the previous proposition, the only thing that we need to do is to see that, in this case, the matrix $M$ writes $\lambda Id$ for some $\lambda \in \R$. 
This will come from  the fact that the heat kernel is a radial function, this means that it only depends on $|x|$ and $|z|$.
Recall also that a $H$-type group can be identified with $\R^{2n+m}$ and that the vector fields $V_i$, $i=1, \dots, 2n$ read:
$$
\hat V_i = \frac{\partial}{\partial x_i} -\frac{1}{2} \sum_{j=1}^m <J_{u_j} x, e_i> \frac{\partial}{\partial z_j} 
$$
where $J_{u_j}$ is an orthogonal map of $\R^{2n}$.
Therefore the coefficient $M_{ij} $ of the matrix $M$ can be written as:
$$
\int_{x\in \R^{2n}} \int_{z\in \R^m} \hat V_i f(|x|^2, |z|^2) \hat V_j f(|x|^2, |z|^2) g(|x|^2, |z|^2) dx dz
$$
for some functions $f$ and $g$.
Now,
$$
V_i f(|x|^2, |z|^2) = 2 x_i \partial_1 f - \sum_{l=1}^m z_l <J_{u_l} x, e_i> \partial_2 f
$$
and the result is coming by expanding the product and noticing that:

$$
\int_{x\in \R^{2n}} x_i x_j h_1(|x|,|z|) dx = \delta_{ij} \tilde h_1(|z|),
$$

$$
\int_{z\in \R^m} z_j h_2(|x|,|z|)dz =0,
$$

$$
\int_{z\in \R^m} z_l z_p h_3(x) =0 \textrm{ for } l\neq p,$$

$$
\int_{x\in \R^{2n}} z_l^2 <J_{u_l} x, e_i> <J_{u_l} x, e_j> h_4(|x|,|z|) dx =  \delta_{ij} z_l^2 \tilde h_4(|z|)
$$
since $J_{u_l} x$ is an orthogonal map of $\R^{2n}$ for $l=1, \dots, m$ 
and that
$
\int_{z\in \R^m} z_l^2 \tilde h_4(|z|) dz 
$ 
does not depend on $l$.

Note that the estimates in \cite{eldredge} show that all the integrals appearing in the above argument are well defined. 
\end{proof}

\section{Isoperimetric inequality}

In this section, we show that the reverse Poincar\'e inequality we proved in the previous section can be used to prove the isoperimetric inequality in Carnot groups.  To this end, we adapt some beautiful
ideas of Varopoulos (see \cite{Varopoulos2}, pp.256-258) and Ledoux
(see pp. 22 in \cite{ledoux-bourbaki}, see also Theorem 8.4 in
\cite{ledoux-stflour})

In order to state the inequality we need the notion of horizontal perimeter, which is defined from the horizontal variation of a function.

 Following \cite{CDG}, given a function $f\in L^1_{loc}(\G)$ we define the horizontal total variation of $f$ by
\[
\text{Var}_{\Ho}(f) = \underset{\phi\in \F(\G)}{\sup} \int_\G
f \left(\sum_{i=1}^d V_i \phi_i\right) d\mu.
\]
where
\[
\mathcal F(\G) = \{\phi\in C^1_0(\G,\mathcal H)\mid
||\phi||_\infty \le 1\}.
\]
Here, for $\phi = \sum_{i=1}^d \phi_i V_i$, we have let
$||\phi||_\infty = \underset{\G}{\sup} \sqrt{\sum_{i=1}^d
\phi_i^2}$. 

The space \[ BV_\Ho(\G) = \{f\in L^1(\G)\mid
\text{Var}_\Ho(f)<\infty\},
\]
endowed with the norm
\[
||f||_{BV_\Ho(\G)} = ||f||_{L^1(\G)} + \text{Var}_\Ho(f),
\]
is a Banach space.  One can note that when $f\in
W^{1,1}_\Ho(\G)$, then $f\in BV_\Ho(\G)$, and one has in fact
\[
\text{Var}_\Ho(f) = ||\sqrt{\Gamma(f)}||_{L^1(\G)}.
\]
Given a measurable set $E\subset \G$ we say that it has finite
horizontal perimeter  if $\mathbf 1_E\in BV_\Ho(\G)$. In
such case we define the horizontal perimeter of $E$  by
\[
P_\Ho(E) = \text{Var}_\Ho(\mathbf 1_E).
\]
We say that a measurable set $E\subset \G$ is a Caccioppoli set if
$P_\mathcal{H}(E)<\infty$. 

\

 We now prove the following result which is due to Varopoulos.

\begin{theorem}[Isoperimetric inequality]\label{T:iso} 
There is a constant $C_{\emph{iso}} >0$, such that for every bounded
Caccioppoli set $E\subset \G$
\[
\mu(E)^{\frac{Q-1}{Q}} \le C_{\emph{iso}} P_\Ho(E).
\]
\end{theorem}

\begin{proof}
Let $f \in C_0^\infty(\G)$. From Proposition \ref{P-I-gen}, we have
\begin{align*}
\Gamma (P_t f)  \le \frac{\Lambda }{t}  \|
f\|^2_{L^\infty(\G)},\ \ \ \ t>0.
\end{align*}
Thus,
\[
\| \sqrt{\Gamma (P_t f)} \|_{L^\infty(\G)} \le
\sqrt{\frac{\Lambda}{t} }\| f\|_{L^\infty(\G)}.
\]
Applying this inequality to $g \in C_0^\infty(\G)$, with $g\ge 0$
and $||g||_{L^\infty(\G)}\le 1$, if $f \in C_0^1(\G)$ we have
\begin{align*}
\int_\G g(f-P_tf) d\mu & = \int_0^t \int_\G g \frac{\p P_sf}{\p s}
d\mu ds = \int_0^t \int_\G g L P_sf d\mu ds  = - \int_0^t \int_\G \Gamma(P_sg,f) d\mu ds
\\
& \le \int_0^t  \| \sqrt{\Gamma(P_sg)} \|_{L^\infty(\G)}\int_\G
\sqrt{\Gamma(f)} d\mu ds  \le 2 \ \sqrt{ \Lambda t}
\int_\G\sqrt{\Gamma(f)} d\mu.
\end{align*}
We thus obtain the following basic inequality: for $f \in
C_0^1(\G)$,
\begin{align}\label{PoincareP_t}
\|P_tf - f\|_{L^1(\G)} \le 2 \ \sqrt{\Lambda t}\ \|
\sqrt{\Gamma(f)} \|_{L^1(G)},\ \ \ t>0.
\end{align}
Suppose now that $E\subset \G$ is a bounded Caccioppoli set. Therefore $\mathbf 1_E\in
BV_\Ho(\G)$. By Theorem 1.14 in \cite{GN}. there exists a sequence
$\{f_n\}_{n\in \mathbb N}$ in $C^\infty_0(\G)$ satisfying 
\begin{itemize}
\item[(i)] $||f_n - 1_E||_{L^1(\G)} \to 0$;
\item[(ii)] $\int_\G \sqrt{\Gamma(f_n)} d\mu \to
\text{Var}_\Ho(f)$.
\end{itemize}

Applying \eqref{PoincareP_t} to $f_n$ we obtain \[ \|P_tf_n -
f_n\|_{L^1(\G)} \le \ 2 \ \sqrt{\Lambda t}\ \| \sqrt{\Gamma(f_n)}
\|_{L^1(\G)} = 2 \ \sqrt{\Lambda t}\ Var_\Ho(f_n),\ \ \ n\in
\mathbb N.
\]
Letting $n\to \infty$ in this inequality, we conclude
\[ \|P_t \mathbf 1_E -
\mathbf 1_E\|_{L^1(\G)} \le   2 \ \sqrt{\Lambda t} \
Var_\Ho(\mathbf 1_E) =  2 \ \sqrt{\Lambda t} \ P_\Ho(E),\ \
\ \ t>0.
\]
Observe now that we have
\[
||P_t \mathbf 1_E - \mathbf 1_E||_{L^1(\G)}  = 2\left(\mu(E) -
\int_E P_t \mathbf 1_E d\mu\right).
\]
On the other hand, we have
\[
\int_E  P_t \mathbf 1_E d\mu  = \int_\bM \left(P_{t/2}\mathbf
1_E\right)^2 d\mu.
\]
We thus obtain
\[
||P_t \mathbf 1_E - \mathbf 1_E||_{L^1(\G)} = 2 \left(\mu(E) -
\int_\G \left(P_{t/2}\mathbf 1_E\right)^2 d\mu\right).
\]
We now note that 
\begin{align*}
\int_\G (P_{t/2} \mathbf 1_E)^2 d\mu & \le \left(\int_E
\left(\int_\G p_{t/2}(x,y)^2
d\mu(y)\right)^{\frac{1}{2}}d\mu(x)\right)^2
\\
& = \left(\int_E p_t(x,x)^{\frac{1}{2}}d\mu(x)\right)^2 \le
\frac{p_1(0)}{t^{Q/2}} \mu(E)^2.
\end{align*}
Combining these equations yields
\[
\mu(E)  \le   \ \sqrt{\Lambda t} \ P_\Ho(E) +
\frac{p_1(0)}{t^{Q/2}} \mu(E)^2,\ \ \ \ t>0.
\]
Minimizing in $t$, we conclude
\[
\mu(E)^{\frac{Q-1}{Q}} \le C P_\Ho(E),
\]
with 
\[ 
C = (1+Q)^{\frac{Q+1}{Q}}  p_1(0)^{\frac{1}{Q}} \frac{ \Lambda }{Q} . \] 
\end{proof}

\section{Boundedness of the Riesz transform}

Besides the isoperimetric inequality, the reverse Poincar\'e for the heat semigroup is also intimately related to the boundedness of the Riesz transform. Actually combining Proposition \ref{P-I-gen} with the results in \cite{BG} which built on \cite{ACDH} leads to the following result.

\begin{proposition}\label{T:equivalence}
Let $1<p<\infty$. There exist constants $A_p, B_p>0$ such that 
\begin{equation}\label{RZsr}
A_p \| (-L)^{1/2} f \|_p \le \| \sqrt{\Gamma(f)} \|_p \le B_p \| (-L)^{1/2} f \|_p,\ \ \ \ \ f \in C_0^\infty(\G),
\end{equation}
\end{proposition}

\begin{proof}
The main ingredient is Theorem 1.3 in \cite{ACDH}. The theorem is stated for Riemannian manifolds, but it is checked in \cite{BG} that the arguments go through in the context of Carnot-Carath\'eodory spaces that satisfy the doubling condition and the Poincar\'e inequality. The only thing to check is the bound
\[
\| \sqrt{\Gamma} e^{tL} \|_{\infty \to \infty} \le \frac{C}{\sqrt{t}},
\]
where $e^{tL}$ is the heat semigroup generated by $L$. This bound is a consequence of our Proposition \ref{P-I-gen}.
\end{proof}

We note that the first proof of the boundedness of the Riesz transform in Carnot groups can be found in \cite{LV}.

\end{document}